\documentclass[nointlimits,11pt,oneside]{amsart}
\usepackage{amssymb,natbib,cases,enumerate}
\usepackage{color}
\usepackage{hyperref}
\usepackage[%
	a4paper,
	total={16cm,23cm},
	left=1cm, top=3cm,
	marginparsep=2pt]
{geometry}

\theoremstyle{plain}
\newtheorem{theorem}{Theorem}[section]

\newtheorem{corollary}[theorem]{Corollary}
\newtheorem{proposition}[theorem]{Proposition}
\theoremstyle{definition}
\newtheorem{remark}[theorem]{Remark}

\numberwithin{equation}{section}

\newcommand{\N}{\mathbb N}
\newcommand{\R}{\mathbb R}
\newcommand{\rn}{\R^n}
\DeclareMathOperator{\spn}{span}
\DeclareMathOperator{\spt}{spt}
\DeclareMathOperator{\cart}{\times}
\DeclareMathOperator{\rank}{rank}
\newcommand{\M}{\mathfrak{M}}
\newcommand{\Mpl}{\M^+}

\renewcommand{\d}{{\fam0 d}}

\usepackage[textsize=tiny,textwidth=4cm,colorinlistoftodos]{todonotes}

\makeatletter
\def\paragraph{\bigskip\@startsection{paragraph}{4}%
  \z@\z@{-\fontdimen2\font}%
  {\normalfont\bfseries}}
\makeatother

\begin{document}

\title{Measure of noncompactness of Sobolev embeddings on strip-like domains}
\author{David E. Edmunds, Jan Lang, Zden\v ek Mihula}

\address{David E. Edmunds, Department of Mathematics, University of Sussex, Falmer, Brighton, BN1~9QH, UK}
\email{davideedmunds@aol.com}
\urladdr{\href{https://orcid.org/0000-0002-8528-3067}{0000-0002-8528-3067}}

\address{Jan Lang, Department of Mathematics, Ohio State University,
Columbus OH, 43210-1174 USA}
\email{lang@math.osu.edu}
\urladdr{\href{https://orcid.org/0000-0003-1582-7273}{0000-0003-1582-7273}}

\address{Zden\v ek Mihula, Charles University, Faculty of Mathematics and
Physics, Department of Mathematical Analysis, Sokolovsk\'a~83,
186~75 Praha~8, Czech Republic}
\email{mihulaz@karlin.mff.cuni.cz}
\urladdr{\href{https://orcid.org/0000-0001-6962-7635}{0000-0001-6962-7635}}


\thanks{This research was partly supported by the Charles University, project GA UK No.~1056119, and by the grant P201-18-00580S of the Grant Agency of the Czech Republic.}

\begin{abstract}
We compute the precise value of the measure of noncompactness of Sobolev embeddings $W_0^{1,p}(D)\hookrightarrow L^p(D)$, $p\in(1,\infty)$, on strip-like domains $D$ 
of the form $\R^k\times\prod\limits_{i=1}^{n-k}(a_i,b_i)$. We show that such embeddings are always maximally noncompact, that is, their measure of noncompactness coincides with their norms. Furthermore, we show that not only the measure of noncompactness but also all strict $s$-numbers of the embeddings in question coincide with their norms. We also prove that the maximal noncompactness of Sobolev embeddings on strip-like domains remains valid even when Sobolev-type spaces built upon general rearrangement-invariant spaces are considered. As a by-product we obtain the explicit form for the first eigenfunction  of the pseudo-$p$-Laplacian on an $n$-dimensional rectangle. 
\end{abstract}

\date{\today}

\maketitle

\setcitestyle{numbers}
\bibliographystyle{plainnat}


\section{Introduction}
Let $T$ be a bounded linear map between Banach spaces $X$ and $Y.$ For each 
$ k\in \mathbb{N}$ the $k^{th}$ entropy number $e_{k}(T)$ of $T$ is defined by 

\[
e_{k}(T)=\inf \left\{ \varepsilon >0:T\left( B_{X}\right) \mbox{ can be
covered by }2^{k-1}\mbox{ balls in }Y \mbox{ with radius }\varepsilon
\right\} ,
\]
where $B_{X}$ is the closed unit ball in $X.$ Since $T$ is compact if and
only if $\lim\limits_{k\rightarrow \infty }e_{k}(T)=0,$ this limit is called the
measure of noncompactness of $T;$ we denote it by $\beta (T).$ Plainly $%
0\leq \beta (T)\leq \left\Vert T\right\Vert ;$ if $\beta (T)=\left\Vert
T\right\Vert $ we say that $T$ is \textit{maximally noncompact. }The
definition of entropy numbers has its roots in the notion of the metric
entropy of a set, introduced by Kolmogorov in the 1930s and which, in its
different variants, has proved useful in numerous branches of mathematics
and theoretical computer science. Sharp upper and lower estimates of $%
e_{k}(T)$ are known in many cases when $T$ is compact:\ information of this
type is useful in connection with the estimation of eigenvalues and is
especially complete when $T$ is an embedding of one function space into
another (see, for example, \cite{ET}).

It is a different story when $T$ is not compact. If $T$ is an embedding map
between function spaces on an open subset $\Omega $ of $\mathbb{R}^{n}$,
possible reasons for noncompactness include

\noindent { (i)} $\Omega $ is unbounded;

\noindent { (ii)} if $\Omega$ is bounded, because of some bad behaviour of the
boundary $\partial \Omega$, or due to particular values of the parameters
involved in the function spaces.

An example of (ii) was provided by Hencl \cite{Hen}, who considered the case
in which $k\in \mathbb{N},p\in \lbrack 1,\infty ),kp<n,1/q=1/p-k/n$ and, in
standard notation, $id\colon W_{0}^{k,p}\left( \Omega \right) \rightarrow
L^{q}\left( \Omega \right) $ is the natural embedding. He showed that $id$
is maximally compact, so that $e_{k}(id)=\left\Vert id\right\Vert $ for all $%
k\in \mathbb{N}.$ Further work in this direction, involving Sobolev spaces
based on Lorentz spaces and maximally noncompact embeddings, is contained in 
\cite{Bou} and \cite{LMOP}.

Little seems to be known in cases of type (i), even in quite basic
situations. For example, suppose that $n=2,$ $\Omega =\mathbb{R}\times
(0,\pi )$ and $I\colon W_{0}^{1,2}\left( \Omega \right) \rightarrow L^{2}\left(
\Omega \right) $ is the natural embedding. Then it is known that $I$ is not
compact, so that $\beta \left( I\right) >0,$ but although this example could
hardly be simpler, the exact value of $\beta \left( I\right) $ appears to be
unknown up to this point. It was shown in \citep[Remark~3.8]{Hen} that the embedding $I$ is maximally noncompact, but the precise value of its measure of noncompactness remained unknown. Here we settle this question by establishing a
more general result in which $n\geq 2,k\in \left\{ 1,...,n-1\right\} ,p\in
(1,\infty ),$ $-\infty <a_{i}<b_{i}<\infty $ for each $i\in \left\{
1,...,n-k\right\} ,$ and 
\[\label{intro:defD}
D=\R^{k}\times \prod_{i=1}^{n-k}\left( a_{i},b_{i}\right) ;
\]%
the norm on $W_{0}^{1,p}\left( D\right) $ is defined by 
\[
\left( \left\Vert u\right\Vert _{L^{p}(D)}^{p}+\left\Vert \left\vert \nabla
u\right\vert _{l^{p}}\right\Vert _{L^{p}(D)}^{p}\right) ^{1/p}.
\]%
We show that the natural embedding $I_{p}:W_{0}^{1,p}\left( D\right)
\rightarrow L^{p}\left( D\right) $ is maximally noncompact and 
\[
\beta \left( I_{p}\right) =\left\Vert I_{p}\right\Vert =\left( 1+\left(
p-1\right) \left( \frac{2\pi }{p\sin (\pi /p)}\right)
^{p}\sum_{i=1}^{n-k}\left( b_{i}-a_{i}\right) ^{-p}\right)^{-1/p}.
\]%
For the particularly elementary illustration involving $I$ that was
mentioned immediately above this gives the attractive formula%
\[
\beta \left( I\right) =\left\Vert I\right\Vert =1/\sqrt{2}.
\]%
It is also shown that the strict $s$-numbers of $I_{p}$ (that is, the
approximation, isomorphism, Gelfand, Bernstein, Kolmogorov and Mityagin
numbers) exhibit the same behaviour as the entropy numbers: the $k^{th}$
such strict $s-$number coincides with $\left\Vert I_{p}\right\Vert $ for all 
$k\in \mathbb{N}.$ The proof of these assertions relies on properties of the
pseudo-$p$-Laplacian and the $p$-trigonometric functions. Furthermore, we show that the embedding of the Sobolev-type space $W_0^1X(D)$ built upon a general rearrangement-invariant space $X(D)$ to a rearrangement-invariant space $Y(D)$ is always maximally noncompact provided that the space $Y(D)$ has absolutely continuous norm. Precise definitions are contained in the following section. 

\section{Background material}
In this section, we fix the notation used throughout this paper and collect the fundamental theoretical background needed later in \hyperref[sec:main]{Section~\ref*{sec:main}}. Let $\Omega\subseteq\rn$ be an open set (throughout this paper, we assume that $n\geq2$).

We denote the set of all continuous functions that are compactly supported in $\Omega$ by $\mathcal C_c(\Omega)$. The set of all smooth (i.e., infinitely differentiable) functions that are compactly supported in $\Omega$ is denoted by $\mathcal C_0^\infty(\Omega)$.

For $p\in[1,\infty)$, $W^{1,p}(\Omega)$ stands for \emph{the classical first-order Sobolev space} on $\Omega$ endowed with the norm
\begin{equation}\label{prel:W1pnorm}
\|u\|_{W^{1,p}(\Omega)}=\left(\|u\|^p_{L^p(\Omega)}+\||\nabla u|_{\ell^p}\|_{L^p(\Omega)}^p\right)^\frac1{p},\ u\in W^{1,p}(\Omega),
\end{equation}
where $|\nabla u|_{\ell^p}$ is the $\ell^p$-norm of the (weak) gradient of $u$, that is,
\begin{equation*}
|\nabla u|_{\ell^p}=\left(\sum_{i=1}^n\left|\frac{\partial u}{\partial x_i}\right|^p\right)^\frac1{p}.
\end{equation*}
We denote the closure of $\mathcal C_0^\infty(\Omega)$ in $W^{1,p}(\Omega)$ by $W_0^{1,p}(\Omega)$.

We shall also work with Sobolev-type spaces built upon function spaces more general than the Lebesgue spaces. We say that a functional $\varrho\colon\Mpl(\Omega)\to[0,\infty]$, where $\Mpl(\Omega)$ is the set of all nonnegative measurable functions on $\Omega$, is a rearrangement-invariant Banach function norm if, for all $f,g,F_j\in\Mpl(\Omega)$, $k\in\N$, for all $\alpha\in[0,\infty)$, and for all measurable $E\subseteq\Omega$, the following properties hold:
\begin{itemize}
\item $\varrho(f)=0$ if and only if $f=0$ a.e., $\varrho(\alpha f)=\alpha\varrho(f)$, $\varrho(f+g)\leq\varrho(f)+\varrho(g)$;
\item if $0\leq g\leq f$ a.e., then $\varrho(g)\leq\varrho(f)$;
\item if $0\leq F_j\nearrow f$ a.e., then $\varrho(F_j)\nearrow\varrho(f)$;
\item if $|E|<\infty$, then $\varrho(\chi_E)<\infty$;
\item if $|E|<\infty$, then there is a constant $C_E$, which may depend only on $E$ and $\varrho$, such that $\int_Ef(x)\,\d x\leq C_E\varrho(f)$;
\item $\varrho(f)=\varrho(g)$ whenever $f$ and $g$ are equimeasurable, that is, $|\{x\in\Omega\colon f(x)>\lambda\}|=|\{x\in\Omega\colon g(x)>\lambda\}|$ for every $\lambda>0$.
\end{itemize}
If $\varrho$ is a rearrangement-invariant Banach function norm, then the set $X(\Omega)=\{f\in\M(\Omega)\colon\varrho(|f|)<\infty\}$, where $\M(\Omega)$ is the set of all measurable functions on $\Omega$, endowed with the norm $\|\cdot\|_{X(\Omega)}$ defined as
\begin{equation*}
\|f\|_{X(\Omega)}=\varrho(|f|),\ f\in X(\Omega),
\end{equation*}
is called a rearrangement-invariant Banach function space. A rearrangement-invariant Banach function space (we shall write just `a \emph{rearrangement-invariant space}') is a Banach space. Textbook examples of rearrangement-invariant spaces are the Lebesgue spaces $L^p$ ($p\in[1,\infty]$), the (two-parametric) Lorentz spaces or the Orlicz spaces. We say that a function $f\in X(\Omega)$ \emph{has absolutely continuous norm} in $X(\Omega)$ if $\lim\limits_{k\to\infty}\|f\chi_{E_k}\|_{X(\Omega)}=0$ for every sequence $\{E_k\}_{k=1}^\infty$ of measurable sets $E_k\subseteq\Omega$ such that $\lim\limits_{k\to\infty}\chi_{E_k}(x)=0$ for a.e.~$x\in\Omega$. We say that $X(\Omega)$ has absolutely continuous norm if every $f\in X(\Omega)$ has absolutely continuous norm in $X(\Omega)$. For example, the Lebesgue space $L^p(\Omega)$ has absolutely continuous norm if and only if $p<\infty$.

Comprehensive accounts of the theory of rearrangement-invariant spaces can be found, e.g., in \citep{BS:88} or \citep{PKJF:13}.

We denote \emph{the (first-order) Sobolev-type space built upon a rearrangement-invariant space} $X(\Omega)$, that is, the set of all weakly differentiable functions from $X(\Omega)$ whose gradients also belong to $X(\Omega)$, by $W^1X(\Omega)$. We equip $W^1X(\Omega)$ with the norm
\begin{equation*}
\|u\|_{W^1X(\Omega)}=\|u\|_{X(\Omega)}+\||\nabla u|_{\ell^1}\|_{X(\Omega)},\ u\in W^1X(\Omega),
\end{equation*}
which turns $W^1X(\Omega)$ into a Banach space. Note that we have $W^{1,p}(\Omega)=W^1L^p(\Omega)$ in the set-theoretical sense, but their norms are merely equivalent (unless $p=1$). The closure of $\mathcal C_0^\infty(\Omega)$ in $W^1X(\Omega)$ is denoted by $W_0^1X(\Omega)$.

Any rule $s\colon T\to\left\{s_m(T)\right\}_{m=1}^\infty$ that assigns to each bounded linear operator $T$ from a Banach space $X$ to a Banach space $Y$ (we shall write $T\in B(X,Y)$) a sequence $\left\{s_m(T)\right\}_{m=1}^\infty$ of nonnegative numbers having, for every $m\in\N$, the following properties:
\begin{itemize}
\item[(S1)] $\|T\|=s_1(T)\geq s_2(T)\geq\cdots\geq0$;
\item[(S2)] $s_m(S+T)\leq s_m(S)+\|T\|$ for every $S\in B(X,Y)$;
\item[(S3)] $s_m(BTA)\leq\|B\|s_m(T)\|A\|$ for every $A\in B(W,X)$ and $B\in B(Y,Z)$, where $W,Z$ are Banach spaces;
\item[(S4)] $s_m(id\colon E\to E)=1$ for every Banach space $E$ with $\dim E\geq m$;
\item[(S5)] $s_m(T)=0$ if $\rank T< m$;
\end{itemize}
is called a \emph{strict $s$-number}. Notable examples of strict $s$-numbers are \emph{the approximation numbers} $a_m$, \emph{the isomorphism numbers} $i_m$, \emph{the Gelfand numbers} $c_m$, \emph{the Bernstein numbers} $d_m$, \emph{the Kolmogorov numbers} $d_m$ or \emph{the Mityagin numbers} $m_n$. For their definitions and the difference between strict $s$-numbers and `standard' $s$-numbers, we refer the reader to \citep[Chapter~5]{EL:11}. In this paper, we will only need the definition of the isomorphism numbers. The $m$-th isomorphism number $i_m(T)$ of $T\in B(X,Y)$ is defined as
\begin{equation*}
i_m(T)=\sup\{\|A\|^{-1}\|B\|^{-1}\}
\end{equation*}
where the supremum is taken over all Banach spaces $G$ with $\dim(G)\geq m$ and all bounded linear operators $A\colon Y\to G$, $B\colon G\to X$ such that $ATB$ is the identity on $G$. The isomorphism numbers are the smallest strict $s$-numbers (\citep[Theorem~3.4]{P:74}), that is,
\begin{equation}\label{prel:isonumssmallest}
s_m(T)\geq i_m(T)
\end{equation}
for every strict $s$-number $s$, for every $T\in B(X,Y)$, and for every $m\in\N$.

If $T\in B(X,Y)$, then its \emph{$m$-th entropy number $e_m(T)$}, $m\in\N$, is defined as
\begin{equation*}
e_m(T)=\inf\left\{\varepsilon>0\colon\ \text{there are $y_1,\dots,y_{2^{m-1}}\in Y$ such that $T(B_X)\subseteq\bigcup_{j=1}^{2^{m-1}}\left(y_j+\varepsilon B_Y\right)$}\right\},
\end{equation*}
where $B_X$ and $B_Y$ are the closed unit balls of $X$ and $Y$, respectively. Note that entropy numbers are not (strict) $s$-numbers (e.g., property (S4) is violated (\citep[Chapter~2, Proposition~1.3]{EE:18})) even though they posses similar properties. Through entropy numbers, we define \emph{the measure of noncompactness $\beta(T)$} of $T\in B(X,Y)$ as
\begin{equation*}
\beta(T)=\lim_{m\to\infty}e_m(T).
\end{equation*}
Note that the limit always exists because the sequence $\{e_m(T)\}_{m=1}^\infty$ is nonincreasing. Furthemore, we have $0\leq\beta(T)\leq\|T\|$, and the operator $T$ is compact if and only if $\beta(T)=0$. We say that the operator $T$ is \emph{maximally noncompact} if $\beta(T)=\|I\|$. Thanks to the monotonicity of $\{e_m(T)\}_{m=1}^\infty$, the operator $T$ is maximally noncompact if and only if $e_m(T)=\|T\|$ for every $m\in\N$.

A great deal of information on (strict) $s$-numbers and entropy numbers can be found in the pioneering work of Pietsch \citep{P:74, P:78}  as well as, e.g., in books \citep{CS:90, EE:18} and references therein.

Lastly, let us briefly recall \emph{the generalized $p$-trigonometric functions}. For $p\in(1,\infty)$, $\sin_p$ is defined on $[0,\frac{\pi_p}{2}]$ as the inverse function to the increasing function
\begin{equation*}
[0,1]\ni t\mapsto\int_0^t\left(1-s^p\right)^{-\frac1{p}}\,\d s,
\end{equation*}
where
\begin{equation*}
\pi_p=2\int_0^1\left(1-s^p\right)^{-\frac1{p}}\,\d s.
\end{equation*}
We extend $\sin_p$ to $[-\pi_p,\pi_p]$ by defining $\sin_p(t)=\sin_p(\pi_p-t)$, $t\in\left[\frac{\pi_p}{2},\pi_p\right]$, and $\sin_p(t)=-\sin_p(-t)$, $t\in[-\pi_p, 0]$.
Finally, we extend $\sin_p$ to the whole real line in such a way that the resulting function is $2\pi_p$-periodic. The function $\sin_p$ is continuously differentiable on $\R$, and its derivative is denoted by $\cos_p$. Note that $\sin_2=\sin$, $\cos_2=\cos$ and $\pi_2=\pi$. The interested reader can find more information on properties of the generalized $p$-trigonometric functions as well as their connection with the theory of the $p$-Laplacian, e.g., in \citep{EL:11, L:95, O:84}.

\section{Noncompactness}\label{sec:main}
Although the following proposition concerning the density of smooth compactly supported functions in rearrangement-invariant spaces having absolutely continuous norm on open sets $\Omega\subseteq\rn$ is folklore, the only reference that we could find is \citep[Lemma~2.10]{KS:14}, which deals with the particular case $\Omega=\R$. For the reader's convenience, we sketch a proof of the assertion in full generality.
\begin{proposition}\label{prop:density}
Let $\Omega\subseteq\rn$ be a (nonempty) open set and let $X(\Omega)$ be a rearrangement-invariant space on $\Omega$. If $X(\Omega)$ has absolutely continuous norm, then smooth compactly supported functions on $\Omega$ are dense in $X(\Omega)$.
\end{proposition}
\begin{proof}
Let $u\in X(\Omega)\setminus\{0\}$. Since $X(\Omega)$ has absolutely continuous norm, bounded functions supported in sets of finite measure are dense in $X(\Omega)$ (e.g.~\citep[Chapter~1, Proposition~3.10 and Theorem~3.13]{BS:88}). Therefore, we may assume, without loss of generality, that $u$ is bounded on $\Omega$. Let $\varepsilon>0$ be given. Set $\Omega_k=\Omega\cap\{x\in\rn\colon |x|<k\}$ for $k\in\N$. Clearly, $\chi_{\Omega\setminus\Omega_k}\to 0$ as $k\to\infty$. Hence, since $u$ has absolutely continuous norm in $X(\Omega)$, there is $k\in\N$ such that
\begin{equation}\label{prop:density:unboundedpart}
\|u\chi_{\Omega\setminus\Omega_k}\|_{X(\Omega)}<\frac{\varepsilon}{3}.
\end{equation}
Moreover, we may assume that $\Omega_k\neq\emptyset$. 

Thanks to Lusin's theorem (recall that $\Omega_k$ is bounded and $u$ is a bounded measurable function),  for each $j\in\N$, there is a compact set $F_j\subseteq\Omega_k$ and a continuous compactly supported function $f_j\in\mathcal C_c(\Omega_k)$ such that
\begin{align}
u&=f_j\quad\text{on $F_j$}\label{prop:density:continuouspart},\\
\sup_{x\in\Omega_k}|f_j(x)|&\leq\sup_{x\in\Omega_k}|u(x)|\label{prop:density:fkboundedbyu},\\
|\Omega_k\setminus F_j|&<\frac1{j}.\label{prop:density:omeganminusfk}
\end{align}
Using the fact that $X(\Omega)$ has absolutely continuous norm again, we can find $j\in\N$ large enough such that
\begin{equation}\label{prop:density:noncontinuouspart}
\|\chi_{\Omega_k\setminus F_j}\|_{X(\Omega)}<\frac{\varepsilon}{6\|u\|_{L^\infty(\Omega)}}
\end{equation}
owing to \eqref{prop:density:omeganminusfk}.
 Furthermore, since $f_j$ is continuous and compactly supported in the open set $\Omega_k$, we can employ a standard mollification argument to find a smooth compactly supported function $g\in\mathcal C^\infty_0(\Omega_k)$ such that
\begin{equation}\label{prop:density:smoothuniformapprox}
\sup_{x\in\Omega_k}|f_j(x)-u(x)|<\frac{\varepsilon}{3\|\chi_{\Omega_k}\|_{X(\Omega)}}
\end{equation}
(note that $0<\|\chi_{\Omega_k}\|_{X(\Omega)}<\infty$ because $\Omega_k$ has finite positive measure).

Finally, combining \eqref{prop:density:unboundedpart}, \eqref{prop:density:continuouspart}, \eqref{prop:density:fkboundedbyu}, \eqref{prop:density:noncontinuouspart} and \eqref{prop:density:smoothuniformapprox}, we arrive at
\begin{align*}
\|u-g\|_{X(\Omega)}&\leq\|u\chi_{\Omega\setminus\Omega_k}\|_{X(\Omega)}+\|u\chi_{\Omega_k}-f_j\|_{X(\Omega)} + \|f_j-g\|_{X(\Omega)}\\
&< \frac{\varepsilon}{3} + \frac{2\|u\|_{L^\infty(\Omega_k)}}{6\|u\|_{L^\infty(\Omega)}}\varepsilon + \frac{\varepsilon}{3} \leq\varepsilon.
\end{align*}
\end{proof}

\begin{remark}
Since the rearrangement invariance of $X(\Omega)$ was not used at all, \hyperref[prop:density]{Proposition~\ref*{prop:density}} is actually valid even when $X(\Omega)$ is just a Banach function space, a function space defined through a functional $\varrho\colon\Mpl(\Omega)\to[0,\infty]$ that has all properties of a rearrangement-invariant Banach function norm but the last one.
\end{remark}

The following theorem shows that the Sobolev embedding $W_0^1X(D)\hookrightarrow Y(D)$ on a strip-like domain $D$ is always maximally noncompact whatever the rearrangement-invariant spaces $X(D)$ and $Y(D)$ are provided that the target space $Y(D)$ has absolutely continuous norm.
\begin{theorem}\label{thm:maxnoncompact}
Let $k\in\{1,\dots,n-1\}$ and $-\infty<a_i<b_i<\infty$, $i=1,\dots,n-k$. Set $D=\R^k\times\prod\limits_{i=1}^{n-k}(a_i,b_i)\subseteq\rn$. Let $X(D)$, $Y(D)$ be rearrangement-invariant spaces on $D$. Assume that $W_0^1X(D)\hookrightarrow Y(D)$. If $Y(D)$ has absolutely continuous norm, then
\begin{equation*}
e_m(W_0^1X(D)\hookrightarrow Y(D))=\|W_0^1X(D)\hookrightarrow Y(D)\|\quad\text{for every $m\in\N$},
\end{equation*}
that is, the canonical embedding $W_0^1X(D)\hookrightarrow Y(D)$ is maximally noncompact.
\end{theorem}
\begin{proof}
Throughout this proof, we denote the canonical embedding $W_0^1X(D)\hookrightarrow Y(D)$ by $I$. Suppose that there is $m\in\N$ such that $e_m(I)<\|I\|$. Let $r,\tilde{r}>0$ be such that $e_m(I)<r<\tilde{r}<\|I\|$. As $r>e_m(I)$, there are functions $g_j\in Y(D)$,  $j=1,\dots,2^{m-1}$, such that
\begin{equation}\label{thm:snumbersstrip:entropy}
\forall u\in W_0^1X(D), \|u\|_{W^1X(D)}\leq1,\exists j\in\{1,\dots,2^{m-1}\}\colon\|u-g_j\|_{Y(D)}\leq r.
\end{equation}
Furthermore, since (smooth) compactly supported functions are dense in $Y(D)$ by \hyperref[prop:density]{Proposition~\ref*{prop:density}}, there are functions $\tilde{g}_j\in\mathcal C^\infty_0(D)$ such that
\begin{equation}\label{thm:snumbersstrip:densityLp}
\|g_j-\tilde{g}_j\|_{Y(D)}<\tilde{r}-r\quad\text{for every $j\in\{1,\dots,2^{m-1}\}$}.
\end{equation}

For every $l>0$, set $D_l=(-l,l)^k\times\prod\limits_{i=1}^{n-k}(a_i,b_i)$. Since $\|I\|=\lim\limits_{l\to\infty}\sup\limits_{\substack{u\in\mathcal C^\infty_0(D_l)\\u\neq0}}\frac{\|u\|_{Y(D)}}{\|u\|_{W^1X(D)}}$ and the functions $\tilde{g}_j$ are compactly supported in $D$, there is $l>0$ such that
\begin{align}
\sup\limits_{\substack{u\in\mathcal C^\infty_0(D_l)\\u\neq0}}\frac{\|u\|_{Y(D)}}{\|u\|_{W^1X(D)}}>\tilde{r}\label{thm:snumbersstrip:almostextremal}\\
\intertext{and}
\bigcup_{j=1}^{2^{m-1}}\spt \tilde{g}_j\subseteq D_l.\label{thm:snumbersstrip:sptofgj}
\end{align}

Combining \eqref{thm:snumbersstrip:almostextremal} with the translation invariance of the $Y(D)$ and $W^1X(D)$ norms in the first $k$-directions, which follows immediately from the rearrangement invariance of the $X(D)$ and $Y(D)$ norms, we see that there is a function $u\in\mathcal C^\infty_0(\tilde{D}_l)$ where $\tilde{D}_l=(l,3l)^k\times\prod\limits_{i=1}^{n-k}(a_i,b_i)$ such that
\begin{equation}\label{thm:snumbersstrip:almostextremaltranslated}
\|u\|_{W^1X(D)}=1\quad\text{and}\quad\|u\|_{Y(D)}>\tilde{r}.
\end{equation}
Finally, using \eqref{thm:snumbersstrip:densityLp}, \eqref{thm:snumbersstrip:sptofgj} and \eqref{thm:snumbersstrip:almostextremaltranslated}, we see that
\begin{align*}
\|u-g_j\|_{Y(D)}&\geq\|u-\tilde{g}_j\|_{Y(D)}-\|\tilde{g}_j-g_j\|_{Y(D)}\\
&\geq\|\left(u-\tilde{g}_j\right)\chi_{\tilde{D}_l}\|_{Y(D)}-\|\tilde{g}_j-g_j\|_{Y(D)}\\
&> \tilde{r}-(\tilde{r}-r)=r
\end{align*}
for every $j=1,\dots,2^{m-1}$, which contradicts \eqref{thm:snumbersstrip:entropy}. Hence $e_m(I)\geq\|I\|$, and so $e_m(I)=\|I\|$.
\end{proof}

\begin{remark}
Note that our choice of the norm on $W^1X(D)$ is immaterial in \hyperref[thm:maxnoncompact]{Theorem~\ref*{thm:maxnoncompact}} and the theorem remains valid even when $W^1X(D)$ is endowed with any equivalent norm. In particular, the assertion of the theorem is also true for the standard Sobolev embeddings $W_0^{1,p}(D)\hookrightarrow L^q(D)$ with either $p\in[1,n)$ and $q\in\left[p,\frac{np}{n-p}\right]$ or $p\in[n,\infty)$ and $q\in[p,\infty)$ (cf.~\citep[Theorem~4.12]{AF:03}).
\end{remark}

In the case where both domain and target space are the Lebesgue space $L^p$, not only do we know that the corresponding Sobolev embedding is maximally noncompact, but we also know the exact value of the norm of the embedding.
\begin{proposition}\label{prop:normonstrip}
Let $p\in(1,\infty)$. Let $k\in\{0,1,\dots,n-1\}$ and $-\infty<a_i<b_i<\infty$, $i=1,\dots,n-k$. Set $D=\R^k\times\prod\limits_{i=1}^{n-k}(a_i,b_i)\subseteq\rn$. The norm of the canonical embedding $W_0^{1,p}(D)\hookrightarrow L^p(D)$ satisfies
\begin{equation}\label{prop:normonstrip:sup}
\|W_0^{1,p}(D)\hookrightarrow L^p(D)\|=\left(1+\pi^p_p(p-1)\sum_{i=1}^{n-k}\frac1{(b_i-a_i)^p}\right)^{-\frac1{p}}.
\end{equation}
\end{proposition}
\begin{proof}
Since
\begin{equation*}
\|W_0^{1,p}(D)\hookrightarrow L^p(D)\|^p=\sup_{\substack{u\in W_0^{1,p}(D)\\u\neq0}}\frac{\|u\|^p_{L^p(D)}}{\|u\|^p_{L^p(D)}+\||\nabla u|_{\ell^p}\|^p_{L^p(D)}} = \sup_{\substack{u\in W_0^{1,p}(D)\\u\neq0}}\frac1{1+\frac{\||\nabla u|_{\ell^p}\|_{L^p(D)}^p}{\|u\|^p_{L^p(D)}}},
\end{equation*}
we clearly have that
\begin{equation}\label{prop:normonstrip:supalt}
\|W_0^{1,p}(D)\hookrightarrow L^p(D)\|=\left(1+\inf_{\substack{u\in W_0^{1,p}(D)\\u\neq0}}\frac{\||\nabla u|_{\ell^p}\|^p_{L^p(D)}}{\|u\|^p_{L^p(D)}}\right)^{-\frac1{p}}.
\end{equation}
Let $\lambda$ denote the infimum in \eqref{prop:normonstrip:supalt}. We shall show that
\begin{equation}\label{prop:normonstrip:lambda}
\lambda = \pi^p_p(p-1)\sum_{i=1}^{n-k}\frac1{(b_i-a_i)^p}.
\end{equation}

Assume that $k>0$. For each $l>0$, we set $D_l=(-l,l)^{k}\times\prod\limits_{i=1}^{n-k}(a_i,b_i)$ and we also define the function $u_l$ as
\begin{equation*}
u_l(x, y)=\left(\prod_{j=1}^k\sin_p\left(\frac{\pi_p x_i}{l}\right)\right)\left(\prod_{i=1}^{n-k}\sin_p\left(\frac{\pi_p(y_i-a_i)}{b_i-a_i}\right)\right),\ (x_1,\dots,x_k,y_1,\dots,y_{n-k})\in D_l,
\end{equation*}
and we extend it outside the rectangle $D_l$ by zero. Since $u_l\in W^{1,p}_0(D_l)$, we have that $u_l\in W_0^{1,p}(D)$. It follows from basic properties of the $p$-trigonometric functions (\citep{L:95}, also~\citep[Chapter~2, (2.22), (2.23)]{EL:11}) and Fubini's theorem that
\begin{align*}
\|u_l\|_{L^p(D)}^p&=\left(\frac{2l}{p}\right)^{k}\prod_{i=1}^{n-k}\frac{b_i-a_i}{p}\quad\text{and}\\
\||\nabla u_l|_{\ell^p}\|_{L^p(D)}^p&=\frac{\pi^p_p}{p'p^{n-k-1}}\left(\frac{2l}{p}\right)^{k}\left(\prod_{i=1}^{n-k}(b_i-a_i)\right)\left(\frac{k}{l^p}+\sum_{i=1}^{n-k}\frac1{(b_i-a_i)^p}\right),
\end{align*}
whence
\begin{equation*}
\lambda\leq\pi^p_p(p-1)\left(\frac{k}{l^p}+\sum_{i=1}^{n-k}\frac1{(b_i-a_i)^p}\right).
\end{equation*}
Hence, since $l>0$ was arbitrary, we obtain that
\begin{equation}\label{prop:normonstrip:upper}
\lambda\leq\pi^p_p(p-1)\sum_{i=1}^{n-k}\frac1{(b_i-a_i)^p}.
\end{equation}

Next, it is well known (\citep[page~28]{O:84}, also~\citep[Theorem~3.3]{EL:11}) that
\begin{equation}\label{prop:normonstrip:firstonedimeigen}
\inf_{\substack{v\in W^{1,p}_0((a,b))\\v\neq0}}\frac{\|v'\|^p_{L^p((a,b))}}{\|v\|^p_{L^p((a,b))}}=\pi^p_p\frac{p-1}{(b-a)^p}\quad\text{whenever $-\infty<a<b<\infty$}.
\end{equation}
Since smooth compactly supported functions are dense in $W_0^{1,p}(D)$, we have that
\begin{equation*}
\lambda=\inf_{\substack{u\in \mathcal C^\infty_0(D)\\u\neq0}}\frac{\||\nabla u|_{\ell^p}\|^p_{L^p(D)}}{\|u\|^p_{L^p(D)}}.
\end{equation*}
Let $u\in \mathcal C^\infty_0(D)$. Since the function $(a_i,b_i)\ni t\mapsto u(x,y_1,\dots,y_{i-1},t,y_{i+1},\dots,y_{n-k})$ is in $\mathcal C^\infty_0((a_i,b_i))$ for each $i\in\{1,\dots,n-k\}$ and every fixed $x\in\R^k$, $y_j\in(a_j,b_j)$, $j\in\{1,\dots,n-k\}\setminus\{i\}$, it follows from \eqref{prop:normonstrip:firstonedimeigen} that
\begin{align*}
&\int_{a_i}^{b_i}\left|\frac{\partial u}{\partial t}(x,y_1,\dots,y_{i-1},t,y_{i+1},\dots,y_{n-k})\right|^p\,\d t\\
&\quad\geq\pi^p_p\frac{p-1}{(b_i-a_i)^p}\int_{a_i}^{b_i}\left|u(x,y_1,\dots,y_{i-1},t,y_{i+1},\dots,y_{n-k})\right|^p\,\d t.
\end{align*}
Hence, thanks to Fubini's theorem, $\|\frac{\partial u}{\partial y_i}\|_{L^p(D)}^p\geq\pi^p_p\frac{p-1}{(b_i-a_i)^p}\|u\|_{L^p(D)}^p$ for every  $i\in\{1,\dots,n-k\}$. Therefore,
\begin{equation*}
\lambda\geq\inf_{\substack{u\in \mathcal C^\infty_0(D)\\u\neq0}}\frac{\left\|\left(\sum\limits_{i=1}^{n-k}\left|\frac{\partial u}{\partial y_i}\right|^p\right)^\frac1{p}\right\|^p_{L^p(D)}}{\|u\|^p_{L^p(D)}}=\inf_{\substack{u\in \mathcal C^\infty_0(D)\\u\neq0}}\frac{\sum\limits_{i=1}^{n-k}\|\frac{\partial u}{\partial y_i}\|^p_{L^p(D)}}{\|u\|^p_{L^p(D)}}\geq\pi^p_p(p-1)\sum_{i=1}^{n-k}\frac1{(b_i-a_i)^p},
\end{equation*}
which combined with \eqref{prop:normonstrip:upper} implies \eqref{prop:normonstrip:lambda}.

Finally, equality \eqref{prop:normonstrip:sup} follows from equalities \eqref{prop:normonstrip:supalt} and \eqref{prop:normonstrip:lambda}.

The case where $k=0$ is actually simpler and can be proved along the same lines, and so we omit its proof. 
\end{proof}

\begin{remark}
By using the identity $\pi_p=\frac{2\pi}{p\sin\left(\frac{\pi}{p}\right)}$ (e.g.~\citep[(2.7)]{EL:11}), the norm of the  canonical embedding $W_0^{1,p}(D)\hookrightarrow L^p(D)$ can be expressed by means of the standard trigonometric functions as
\begin{equation*}
\|W_0^{1,p}(D)\hookrightarrow L^p(D)\|=\left(1+(p-1)\left(\frac{2\pi}{p\sin\left(      \frac{\pi}{p}\right)}\sum_{i=1}^{n-k}\frac1{(b_i-a_i)^p}\right)^p\right)^{-\frac1{p}}.
\end{equation*}
\end{remark}

\begin{corollary}\label{cor:maximazeronrectangle}
Let $p\in(1,\infty)$ and $R=\prod\limits_{i=1}^n(a_i,b_i)$ where $-\infty<a_i<b_i<\infty$. The norm of the canonical embedding $W_0^{1,p}(R)\hookrightarrow L^p(R)$ is attained by the function $u$ defined as
\begin{equation}\label{cor:maximizer}
u(x)=\prod_{i=1}^{n}\sin_p\left(\frac{\pi_p(x_i-a_i)}{b_i-a_i}\right),\ x\in R.
\end{equation}
Moreover, $u$ is the unique positive maximizer up to a positive multiplicative constant.
\end{corollary}
\begin{proof}
Since a function $u\in W_0^{1,p}(R)$ maximizes $\sup\limits_{\substack{u\in W_0^{1,p}(R)\\u\neq0}}\frac{\|u\|^p_{L^p(R)}}{\|u\|^p_{L^p(R)}+\||\nabla u|_{\ell^p}\|^p_{L^p(R)}}$ if and only if it minimizes $\inf\limits_{\substack{u\in W_0^{1,p}(R)\\u\neq0}}\frac{\||\nabla u|_{\ell^p}\|^p_{L^p(R)}}{\|u\|^p_{L^p(R)}}$, the fact that the function defined by \eqref{cor:maximizer} is a positive maximizer follows immediately from the proof of \hyperref[prop:normonstrip]{Proposition~\ref*{prop:normonstrip}}. The uniqueness (up to a positive multiplicative constant) of the positive minimizer of the Rayleigh quotient $$\inf\limits_{\substack{u\in W_0^{1,p}(D)\\u\neq0}}\frac{\||\nabla u|_{\ell^p}\|^p_{L^p(R)}}{\|u\|^p_{L^p(R)}}$$ was proved in \citep[Lemma~2.1]{BK:04}.
\end{proof}

\begin{remark}
It can be routinely shown that the extreme function for the Rayleigh quotient $\inf\limits_{\substack{u\in W_0^{1,p}(D)\\u\neq0}}\frac{\||\nabla u|_{\ell^p}\|^p_{L^p(R)}}{\|u\|^p_{L^p(R)}}$ is the first eigenvalue of the
pseudo-$p$-Laplacian operator with Dirichlet boundary conditions, i.e.:
\begin{equation}\label{pseudopLaplacian} \tilde{\Delta}_p u = \tilde{\lambda_p} |u|^{p-2}u, \quad \mbox{with } u=0 \mbox{ on } \partial R,
\end{equation}
where 
\[\tilde{\Delta}_p u = \sum_{i=1}^n \frac{\partial}{\partial x_i} \left( 
\left|\frac{\partial u}{\partial x_i} \right|^{p-2} \frac{\partial u}{\partial x_i} \right).
\]
Then it follows from \hyperref[cor:maximazeronrectangle]{Corollary~\ref*{cor:maximazeronrectangle}} that the 
first eigenfunction for the Dirichlet problem \eqref{pseudopLaplacian} for the pseudo-p-Laplacian operator on the domain $R$ is the function defined by \eqref{cor:maximizer}.  In the case where $R$ is a cube, this was already observed in \citep[Example~2.4]{BK:04}.

This coincides well with the classical result for $p=2$. However, whether all eigenfunctions for \eqref{pseudopLaplacian} on the domain $R$ are of the form
\[\prod_{i=1}^{n}\sin_p\left(\frac{\pi_pk_i(x_i-a_i)}{b_i-a_i}\right),\ (x_1,\dots,x_n)\in R,\ \text{for some $k_i \in \N$},
\]
remains an open question if $p\neq2$.
\end{remark}

Not only is the canonical embedding $W_0^{1,p}(D)\hookrightarrow L^p(D)$ maximally noncompact, but also all its strict $s$-numbers coincide with the norm of the embedding.
\begin{theorem}
Let $p\in(1,\infty)$. Let $k\in\{1,\dots,n-1\}$ and $-\infty<a_i<b_i<\infty$, $i=1,\dots,n-k$. Set $D=\R^k\times(a_1,b_1)\times\prod\limits_{i=1}^{n-k}(a_i,b_i)\subseteq\rn$. We have that
\begin{equation}\label{thm:snumbersstrip:eq}
\begin{aligned}
a_m(I)&=b_m(I)=c_m(I)=d_m(I)=i_m(I)=m_n(I)\\
&=e_m(I)=\|I\|=\left(1+\pi^p_p(p-1)\sum_{i=1}^{n-k}\frac1{(b_i-a_i)^p}\right)^{-\frac1{p}}
\end{aligned}
\end{equation}
for every $m\in\N$, where $I$ stands for the canonical embedding $W_0^{1,p}(D)\hookrightarrow L^p(D)$.

In particular,
\begin{equation*}
s_m(I)=\|I\|=\left(1+\pi^p_p(p-1)\sum_{i=1}^{n-k}\frac1{(b_i-a_i)^p}\right)^{-\frac1{p}}
\end{equation*}
for each strict $s$-number $s$ and every $m\in\N$, and the canonical embedding $W_0^{1,p}(D)\hookrightarrow L^p(D)$ is maximally noncompact.
\end{theorem}
\begin{proof}
The last two equalities in \eqref{thm:snumbersstrip:eq} were already proved in \hyperref[thm:maxnoncompact]{Theorem~\ref*{thm:maxnoncompact}} and \hyperref[prop:normonstrip]{Proposition~\ref*{prop:normonstrip}}, respectively. As for the other equalities,  it suffices to show that $i_m(I)\geq\|I\|$ for all $m\in \N$ thanks to the fact that the isomorphism numbers are the smallest strict $s$-numbers (see \eqref{prel:isonumssmallest}) and property (S1).

We recall that $i_m(I)=\sup\{\|A\|^{-1}\|B\|^{-1}\}$ where the supremum is taken over all Banach spaces $G$ with $\dim(G)\geq m$ and all bounded linear operators $A\colon L^p(D)\to G$, $B\colon G\to W_0^{1,p}(D)$ such that $AIB$ is the identity on $G$.

Let $\varepsilon>0$ be given. Since smooth compactly supported functions are dense in  $W_0^{1,p}(D)$, there are $l>0$ and a function $u\in\mathcal C^\infty_0(D_l)$ such that
\begin{equation}\label{thm:snumbersstrip:almostmaximizer}
\|u\|_{L^p(D)}>\|I\|-\varepsilon\quad\text{and}\quad\|u\|_{W^{1,p}(D)}=1
\end{equation}
where $D_l=(-l,l)^k\times\prod\limits_{i=1}^{n-k}(a_i,b_i)$. For $i=1,2,\dots,m$, we define
\begin{align*}
D^i_l&=((2i-3)l, (2i-1)l)^k\times\prod\limits_{i=1}^{n-k}(a_i,b_i)\\
\intertext{and}
u_i(x,y)&=u(x_1-2(i-1)l, \dots,x_k-2(i-1)l, y),\ (x,y)\in D^i_l\subseteq\R^k\cart\R^{n-k}.
\end{align*}
Clearly, for every $i=1,2,\dots,m$,
\begin{align}
\text{the rectangles $D^i_l$ are mutually disjoint and $u_i\in\mathcal C^\infty_0(D^i_l)$},\label{thm:snumbersstrip:uidisjsup}\\
\||\nabla u_i|_{\ell^p}\|_{L^p(D)}=\||\nabla u|_{\ell^p}\|_{L^p(D)}\ \text{and}\ \|u_i\|_{L^p(D)}=\|u\|_{L^p(D)}\label{thm:snumbersstrip:uisamenorms}.
\end{align}

Now, we define $B\colon\ell^p(\R^m)\to W_0^{1,p}(D)$ by
\begin{equation*}
B(\{\alpha_i\}_{i=1}^m)=\sum_{i=1}^m\alpha_iu_i.
\end{equation*}
The operator $B$ is a well-defined linear operator and $\|B\|=1$. Indeed, $\sum\limits_{i=1}^m\alpha_iu_i\in W_0^{1,p}(D)$ and
\begin{equation}\label{thm:snumbersstrip:Bnorm}
\|B(\{\alpha_i\}_{i=1}^m)\|_{W^{1,p}(D)}^p=\left\|\sum_{i=1}^m\alpha_i u_i\right\|_{W^{1,p}(D)}^p=\sum_{i=1}^m|\alpha_i|^p\|u_i\|_{W^{1,p}(D)}^p=\sum_{i=1}^m|\alpha_i|^p
\end{equation}
owing to \eqref{thm:snumbersstrip:almostmaximizer}, \eqref{thm:snumbersstrip:uidisjsup} and \eqref{thm:snumbersstrip:uisamenorms}.

Next, note that the functions $u_i$, $i=1,\dots,m$, are linearly independent in $L^p(D)$ because they have mutually disjoint supports. Hence, for each $i\in\{1,\dots,m\}$, the linear functionals $\tilde{\gamma}_i\colon\spn\{u_1,\dots,u_m\}\to\R$ defined as
\begin{equation*}
\tilde{\gamma}_i\left(\sum_{j=1}^m\beta_j u_j\right)=\beta_i
\end{equation*}
are well defined. Moreover,
\begin{equation*}
\left|\tilde{\gamma}_i\left(\sum_{j=1}^m\beta_j u_j\right)\right|=|\beta_i|=\frac{|\beta_i|\|u_i\|_{L^p(D)}}{\|u\|_{L^p(D)}}\leq\frac1{\|u\|_{L^p(D)}}\left\|\sum_{j=1}^m\beta_j u_j\right\|_{L^p(D)}
\end{equation*}
thanks to \eqref{thm:snumbersstrip:uidisjsup} and \eqref{thm:snumbersstrip:uisamenorms}. Therefore, by virtue of the Hanh--Banach theorem, there are functionals $\gamma_i\colon L^p(D)\to\R$, $i=1,\dots,m$, such that $\gamma_i=\tilde{\gamma}_i$ on $\spn\{u_1,\dots,u_m\}$ and
\begin{equation}
\|\gamma_i\|\leq\frac1{\|u\|_{L^p(D)}}.\label{thm:snumbersstrip:coordfuncnorm}
\end{equation}
We define $A\colon L^p(D)\to \ell^p(\R^m)$ as
\begin{equation*}
Av=\left(\gamma_1\left(v\chi_{D_l^1}\right),\dots,\gamma_m\left(v\chi_{D_l^m}\right)\right).
\end{equation*}
The operator $A$ is clearly linear, for the functionals $\gamma_i$ are linear. Furthermore
\begin{equation}\label{thm:snumbersstrip:Anorm}
\|A\|\leq\frac1{\|u\|_{L^p(D)}}.
\end{equation}
Indeed,
\begin{equation*}
\|Av\|_{\ell^p(\R^m)}^p=\sum_{i=1}^m\left|\gamma_i\left(v\chi_{D_l^i}\right)\right|^p\leq\frac1{\|u\|^p_{L^p(D)}}\sum_{i=1}^m\|v\chi_{D_l^i}\|_{L^p(D)}^p\leq\frac1{\|u\|^p_{L^p(D)}}\|v\|_{L^p(D)}^p
\end{equation*}
in view of  \eqref{thm:snumbersstrip:coordfuncnorm} and \eqref{thm:snumbersstrip:uidisjsup}.

Finally, upon observing that $AIB$ is the identity on $\ell^p(\R^m)$ because $u_i\chi_{D_l^i}=u_i$ for every $i\in\{1,\dots,m\}$ thanks to \eqref{thm:snumbersstrip:uidisjsup}, we see that
\begin{equation*}
i_m(I)\geq\|A\|^{-1}\|B\|^{-1}\geq\|u\|_{L^p(D)}>\|I\|-\varepsilon
\end{equation*}
owing to \eqref{thm:snumbersstrip:Bnorm}, \eqref{thm:snumbersstrip:Anorm} and \eqref{thm:snumbersstrip:almostmaximizer}, whence $i_m(I)\geq\|I\|$ since $\varepsilon>0$ may be chosen arbitrarily small.
\end{proof}

\paragraph{Acknowledgments}
The third author would like to express his sincere gratitude to the Fulbright Program for supporting him and giving him the opportunity to visit the second author at the Ohio State University as a Fulbrighter and to conduct research with him.

\bibliography{main}

\end{document}